\documentclass[a4paper,reqno]{amsart}

\usepackage{graphicx}
\usepackage{amsmath}
\usepackage{amsthm}
\usepackage{amssymb}
\usepackage{enumitem}
\usepackage[all]{xy}
\usepackage{hyperref}

\numberwithin{equation}{section}
\setlist[1]{leftmargin=*,label={\rm(\arabic*)}}

\theoremstyle{plain}

\newtheorem{thm}{Theorem}[section]
\newtheorem{prop}[thm]{Proposition}
\newtheorem{cor}[thm]{Corollary}
\newtheorem{lem}[thm]{Lemma}
\theoremstyle{definition}

\theoremstyle{remark}
\newtheorem{rmk}[thm]{Remark}

\renewcommand{\top}{\operatorname{top}}

\DeclareMathOperator{\soc}{soc}

\DeclareMathOperator{\rank}{rank}
\DeclareMathOperator{\pd}{pd}
\DeclareMathOperator{\id}{id}
\DeclareMathOperator{\diag}{diag}

\newcommand{\lto}{\longrightarrow}
\mathchardef\mhyphen="2D

\newcommand{\op}{\mathrm{op}}

\newcommand{\inj}{\mathrm{inj}}

\newcommand{\bfC}{\mathbf{C}}

\newcommand{\bfP}{\mathbf{P}}
\newcommand{\bfQ}{\mathbf{Q}}

\newcommand{\bfX}{\mathbf{X}}
\newcommand{\bfY}{\mathbf{Y}}

\newcommand{\bbN}{\mathbb{N}}
\newcommand{\bbZ}{\mathbb{Z}}

\renewcommand{\S}{\mathcal{S}}

\title[Homological properties of Nakayama algebras]
      {A note on homological properties of Nakayama algebras}
\author[Dawei Shen]{Dawei Shen}
\address{Department of Mathematics \\
         Shanghai Key Laboratory of PMMP \\
         East China Normal University \\
         Shanghai 200241 \\
         P. R. China}
\email{dwshen@math.ecnu.edu.cn}
\date{\today}

\begin{document}
\begin{abstract}
Using the resolution quiver for a connected Nakayama algebra,
a fast algorithm is given
to decide whether its global dimension is finite or not
and whether it is Gorenstein or not.
The latter strengthens a result of Ringel.
\end{abstract}
\maketitle

\section{Introduction} \label{sec:1}
Let $A$ be a connected Nakayama algebra.
Following \cite{Rin2013,Gus1985,She2014},
its {\em resolution quiver} is defined as follows:
the vertex set is the set of non-isomorphic simple $A$-modules;
there is a unique arrow
from each simple $A$-module $S$ to $\gamma(S)=\tilde{\tau}\soc{P(S)}$.
Here, $P(S)$ is the projective cover of $S$
and `soc' is the socle of a module.
If $A$ has a simple projective module,
denote by $S_\inj$ the unique simple injective $A$-module up to isomorphism.
Then
\[\tilde{\tau}(S) =
\begin{cases}
\tau(S), & \mbox{ if $S$ is not projective}  \\
S_\inj, & \mbox{ otherwise}
\end{cases}\]
for each simple $A$-module $S$,
where $\tau$ is the Auslander-Reiten translation \cite{ARS1995}.

Let $A$ be a connected Nakayama algebra.
Denote by $R(A)$ its resolution quiver.
It is known that each connected component of $R(A)$ has a unique cycle.
For a cycle $C$ in $R(A)$ with vertices $S_1,S_2,\cdots,S_m$,
the {\em weight} $w(C)$ of $C$ is  $\sum_{k=1}^m\frac{c_k}{n}$.
Here, $n$ is the number of non-isomorphic simple $A$-modules
and $c_k$ is the composition length of $P(S_k)$.
It turns out that $w(C)$ is an integer
and all cycles in $R(A)$ have the same weight \cite{She2014}.
The weight $w(C)$ is called the weight of the algebra $A$.

The resolution quiver is a very efficient tool
for investigating the homological properties of Nakayama algebras.
The Gorenstein projective modules for Nakayama algebras
are described by resolution quivers \cite{Rin2013}.
Resolution quivers are also used
to study the singularity categories of Nakayama algebras \cite{She2015}.

The resolution quiver of a connected Nakayama algebra gives a fast algorithm
to decide whether it is Gorenstein  or not
and whether it is CM-free  or not  \cite{Rin2013}.
In this paper,
we show that the resolution quiver of a connected Nakayama algebra
also gives a fast algorithm
to decide whether its global dimension is finite  or not.

More precisely, we have the following.

\begin{prop} \label{prop:1.1}
Let $A$ be a connected Nakayama algebra.
Then $A$ has finite global dimension if and only if
its resolution quiver  is connected and its weight  is $1$.
\end{prop}

As a consequence of Proposition ~\ref{prop:1.1},
the resolution quiver is connected
for a connected Nakayama algebra with finite global dimension.

For a connected Nakayama algebra,
recall from \cite{Rin2013} that a cycle  in its resolution quiver is called  \emph{black}
provided that the projective dimension of each simple module on this cycle is
not equal to $1$.

The following result strengthens \cite[Proposition 5(a)]{Rin2013}.

\begin{prop} \label{prop:1.2}
Let $A$ be a connected Nakayama algebra with infinite global dimension.
Then $A$ is a Gorenstein algebra if and only if all cycles in
its resolution quiver  are black.
\end{prop}

Let $A$ be a connected  Nakayama algebra.
Take  a complete set $\{S_1,S_2,\cdots,S_n\}$ of pairwise non-isomorphic simple $A$-modules.
The {\em Cartan matrix} $\bfC_A$ of $A$ is an $n \times n$ matrix $(c_{ij})$,
where $c_{ij}$ is the number of copies of $S_i$
appearing in a composition series for the projective cover of $S_j$.
Denote by $\bfC_A^T$ the transpose of $\bfC_A$.

Denote by $c$ the number of cycles and
by $b$ the number of black cycles in the resolution quiver of $A$.

The following result gives a connection
between Cartan matrices and resolution quivers.

\begin{prop} \label{prop:1.3}
Let $A$ be a connected Nakayama algebra.
\begin{enumerate}
  \item The rank of $\bfC_A$ is $n + 1 - c$.
  \item If $b$ is nonzero, then the rank of $(\bfC_A,\bfC_A^T)$ is $n + 1 - b$.
\end{enumerate}
\end{prop}

The paper is organised as follows.
The proofs  of Proposition ~\ref{prop:1.1} and Proposition ~\ref{prop:1.2}
are given in Section ~\ref{sec:2} and  Section ~\ref{sec:3}, respectively.
In Section ~\ref{sec:4}, we study the connection between Cartan matrices and
resolution quivers for Nakayama algebras and prove Proposition ~\ref{prop:1.3}.

\section{Retractions and resolution quivers} \label{sec:2}
Let $A$ be a connected Nakayama algebra.
Denote by $n=n(A)$ the number of non-isomorphic
simple $A$-modules.
Take a sequence  $(S_1,S_2,\cdots,S_n)$ of pairwise non-isomorphic simple $A$-modules
such that the radical  of $P_i$ is a factor module of $P_{i+1}$ for $1\leq i \leq n-1$
and the radical  of $P_n$  is a factor module of $P_{1}$.
Here, $P_i$  is the projective cover of $S_i$.
Denote by $c_i$ the composition length of $P_i$.
The {\em admissible sequence} for $A$ is given by $\mathbf{c}(A)=(c_1,c_2,\cdots,c_n)$;
see \cite[Chapter IV. 2]{ARS1995}.

Following \cite{Gus1985}, there exists a map
$f_A\colon\{1,2,\cdots,n\}\to\{1,2,\cdots,n\}$
such that $n$ divides $f_A(i)-(c_i+i)$.
For $1\leq i \leq n$, we have $\gamma (S_i)=S_{f_A(i)}$.
Then for $1\leq i,j\leq n$,
there is an arrow $S_i \to S_j$ in the resolution quiver of $A$
if and only if $f_A(i)=j$.

Suppose now that $A$ is not selfinjective.
If  $A$  has no simple projective modules,
after possible cyclic permutations,
we may assume that
its admissible sequence is {\em normalized} \cite{CY2014}, that is,
$p(A)=c_1=c_n-1$.
Here, $p(A)$ is the minimal integer among  $c_i$.
For convenience,
if $A$ has a simple projective module, its admissible sequence is always normalized.

Following \cite{CY2014}, there exists a {\em left retraction} $\eta\colon A \to L(A)$,
where $L(A)$ is a connected Nakayama algebra with admissible sequence
$\mathbf{c}(L(A))=(c_1',c_2',\cdots,c_{n-1}')$ such that
$c_i'=c_i-\left[\frac{c_i+i-1}{n}\right]$ for $1\leq i\leq n-1$.
In particular, $n(L(A))=n(A)-1$.
Here, $[x]$ is the largest integer not greater than a real number $x$.
The corresponding sequence of simple $L(A)$-modules is denoted by $(S_1',S_2',\cdots,S_{n-1}')$.

We need the map $\pi\colon\{1,2,\cdots,n\}\to\{1,2,\cdots,n-1\}$
such that $\pi(i)=i$ for $1\leq i\leq n-1$ and $\pi(n)=1$.

Recall from \cite[Lemma 2.1]{She2014} the following result.
For the convenience of the reader, we give a proof here.

\begin{lem} \label{lem:retra}
Let $A$ be a connected Nakayama algebra which is not selfinjective.
Then $\pi f_A(i)=f_{L(A)}(i)$ for $1\leq i\leq n-1$.
\end{lem}

\begin{proof}
Let $f_A(i)=j$ and $c_i+i=kn+j$ with $k\in\bbN$.
For $1\leq i \leq n-1$, we have
\begin{equation} \label{eq:2.1}
c_{i}^\prime+i=c_i+i-\left[\frac{c_i+i-1}{n}\right]=kn+j-\left[\frac{kn+j-1}{n}\right]=k(n-1)+j.
\end{equation}
It follows that $\pi f_A(i)=\pi(j)=f_{L(A)}(i)$.
\end{proof}

Denote by $\gamma'$ the map $\tilde{\tau}\soc P(-)$ for $L(A)$.
It follows from Lemma ~\ref{lem:retra}
that $\gamma'(S_i')= S_{\pi f_A(i)}'$ for $1\leq i\leq n-1$.
Hence  the resolution quiver  of $L(A)$ can be obtained from the resolution quiver of $A$ just
by ``merging'' the vertices  $S_1$ and $S_n$.

Observe that $\gamma(S_n) = \gamma(S_1)$
if $A$ has no simple projective modules
and $\gamma(S_n) = S_1$ if $A$ has a simple projective module.
In particular,
the vertices $S_1$ and $S_n$ lie
on the same connected component in the resolution quiver $R(A)$ of $A$.
Then $R(A)$ and $R(L(A))$ have the same number of connected components.
It follows that they have the same number of cycles
since each connected component of a resolution quiver has a unique cycle.

Let $C$ be a cycle with vertices $S_{x_1},S_{x_2},\cdots,S_{x_s}$ in $R(A)$.
The weight $w(C)$ of $C$ is  $\sum_{k=1}^s\frac{c_{x_k}}{n(A)}$
and the size of $C$ is the number of vertices on $C$.

The following result strengthens \cite[Lemma 2.2]{She2014},
where the  connected Nakayama algebra  is required to have no simple projective modules.

\begin{lem} \label{lem:bij}
Let $A$ be a connected Nakayama algebra which is not selfinjective.
Then there exists a weight preserving bijection between the set of cycles in $R(A)$
and the set of cycles in $R(L(A))$.
Moreover, if $A$ has no simple projective modules or
the simple projective $A$-module  does not lie on a cycle in $R(A)$,
then the bijection also preserves the size.
\end{lem}

\begin{proof}
If $A$ has no simple projective modules,
then the bijection follows from \cite[Lemma 2.2]{She2014}.
We may assume that $A$ has a simple projective module $S_n$.

Let $C$ be a cycle with vertices $S_{x_1},S_{x_2},\cdots,S_{x_s}$ in $R(A)$.
Assume $x_{i+1}=f_A(x_i)$.
Here, we identify $x_{s+1}$ with $x_1$.
Let $c_{x_i}+x_i=k_i n+x_{i+1}$ with $k_i\in\bbN$.
Then
\[w(C)=\frac{\sum_{i=1}^sc_{x_i}}{n}=\sum_{i=1}^s{k_i}.\]
It follows from \eqref{eq:2.1} that for $x_i <n$, we have
$c_{x_i}^\prime+x_i= k_i(n-1)+x_{i+1}$.

There exist two cases:

Case $1$: $S_n$ does not lie on $C$.
Then $S_{x_1}',S_{x_2}',\cdots,S_{x_s}'$ form a cycle $C'$ in $R(L(A))$.

Observe that $c_{x_i}^\prime+x_i= k_i(n-1)+x_{i+1}$
for $1\leq i\leq s$.
Then
\[\sum_{i=1}^sc_{x_i}'=(n-1)\sum_{i=1}^sk_i,\]
Therefore, $w(C)=w(C')$.

Case $2$: $S_n$ lies on $C$.
Since the admissible sequence of $A$ is normalized, we have ${x_s}=n$ and ${x_1}=1$.
It follows that $S_{x_1}',S_{x_2}',\cdots,S_{x_{s-1}}'$ form a cycle $C''$ in $R(L(A))$.

Observe that $k_s=1$ and $c_{x_i}^\prime+x_i= k_i(n-1)+x_{i+1}$
for $1\leq i\leq s-1$.
Then
\[\sum_{i=1}^{s-1}c_{x_i}'=(n-1)\sum_{i=1}^{s-1}k_i+n-1=(n-1)\sum_{i=1}^{s}k_i.\]
Therefore, $w(C)=w(C'')$.

We have shown that there exists an injective weight preserving map
from the set of cycles in $R(A)$ to the set of cycles in $R(L(A))$.
Since $R(L(A))$ and $R(A)$ have the same number of  cycles,
this map is also surjective.
This finishes our proof.
\end{proof}

Recall from \cite[Theorem 3.8]{CY2014} that there exists a sequence of left retractions
\begin{equation} \label{eq:2.2}
A=A_0\overset{\eta_0}\lto A_1\overset{\eta_1}\lto A_2\lto\cdots\lto A_{r-1}\overset{\eta_{r-1}}\lto A_r
\end{equation}
such that each $A_i$ is a connected Nakayama algebra, each
$\eta_i\colon A_i\to A_{i+1}$ is a left retraction and $A_r$ is selfinjective;
the global dimension of $A$ is finite if and only if $A_r$ is simple.

For a connected Nakayama  $A$,
denote by $c(A)$  the number of cycles
and by $w(A)$ the weight of a cycle  in $R(A)$.
We mention that $w(A)$ is an integer and
all cycles in $R(A)$ have the same weight; see \cite{She2014}.

We now  prove Proposition ~\ref{prop:1.1}.

\begin{proof}[\bf Proof of Proposition ~\ref{prop:1.1}]
Applying Lemma ~\ref{lem:bij} to \eqref{eq:2.2}  repeatedly,
we obtain $c(A) = c(A_r)$ and $w(A) = w(A_r)$.
Observe that  $A_r$ is simple if and only if $c(A_r) = 1$  and $w(A_r) = 1$.
Then the global dimension of $A$ is finite if and only if $c(A) = 1$ and $w(A) = 1$.
Each connected component of a resolution quiver has a unique cycle.
Then $c(A) = 1$ if and only if  the resolution quiver of $A$ is connected.
\end{proof}

\section{Two maps on simple modules} \label{sec:3}
Let $A$ be a connected Nakayama algebra.
For an $A$-module $M$,
denote by $\pd  M$ its projective dimension and
by $\id  M$  its injective dimension.
Recall that the {\em syzygy} $\Omega(M)$ of $M$ is
the kernel of its projective cover $p\colon P(M)\lto M$.
Dually,  the {\em cosyzygy} $\Omega^{-1}(M)$ of $M$ is
the cokernel of its injective envelope $i\colon M\lto I(M)$.

It is known that $A$ has a simple projective module
if and only if it has a simple injective module.
In this case, denote by $S_\inj$
the unique simple  injective $A$-module up to isomorphism.
Then there exists map $\tilde\tau$   defined by
\[\tilde{\tau}(S) =
\begin{cases}
\tau(S), & \mbox{ if $S$ is not projective} \\
S_\inj, & \mbox{ otherwise}
\end{cases}\]
for each simple $A$-module $S$,
where $\tau$ the Auslander-Reiten translation \cite{ARS1995}.

Take a complete set $\S$ of pairwise non-isomorphic simple $A$-modules.
Recall from \cite{Gus1985,Rin2013} that
there exist two maps $\gamma,\psi\colon \S \to\S$ given by
\[\gamma(S)=\tilde\tau\soc{P(S)} \mbox{ and } \psi(S)=\tilde\tau^{-1}\top I(S).\]
for each $S$ in $\S$.
Here,  `soc' is the socle and `$\top$' is the top of a module.

The map $\gamma$ determines the resolution quiver for $A$.
Denote by $A^\op$ the opposite algebra of $A$ and by
$\gamma^\op$ the map $\tilde\tau\soc{P(-)}$ for $A^\op$.
Then $D\psi(S)=\gamma^\op(DS)$ for each $S$ in $\S$,
where $D$ is the usual dual for finitely generated $A$-modules.
Hence the resolution quiver of $A^\op$ is
isomorphic to the quiver determined by the map $\psi$.

The following terminology is taken from  \cite{Rin2013}.
A simple $A$-module $S$ is called \emph{$\gamma$-black}
provided that  $\pd S$ is not equal to $1$;
it is called \emph{$\gamma$-cyclic} provided that $\gamma^m(S)=S$ for some integer $m>0$.
Dually, one can define $\psi$-black and $\psi$-cyclic simple $A$-modules.

We need the following lemma.

\begin{lem} \label{lem:mad}
Let $S$ and $T$ be simple $A$-modules.
\begin{enumerate}
  \item If $S$ is not projective, then $\psi(T)=S$ if and only if $T$ is a composition factor in $\Omega^{2}(S)$.
  \item If $\pd  \psi(S)$ is odd, then $\pd  S$ is odd and $\pd  S \leq \pd  \psi(S)-2$.
  \item $S$ is  $\psi$-cyclic if and only if $\pd  S$ is not odd.
  \item $A$ has infinite global dimension if and only if
   the set of $\psi$-cyclic  simple $A$-modules is exactly the set of simple $A$-modules of infinite projective dimension.
\end{enumerate}
\end{lem}

\begin{proof}
If $A$ has no simple projective modules,
then the arguments follow from \cite[Section 3]{Mad2005}.
We mention that they are valid for any connected Nakayama algebra.
\end{proof}

We need the following.

\begin{lem} \label{lem:rin}
Let $M$ be an indecomposable $A$-module.
Then either $\id  M \leq 1$ or $\soc\Omega^{-2}(M)=\psi(\soc M)$.
\end{lem}

\begin{proof}
This is dual to \cite[Lemma 2]{Rin2013}; see also \cite{Gus1985}.
\end{proof}

The following lemma provides a connection
between the maps $\gamma$ and $\psi$.

\begin{lem} \label{lem:bla-mod}
A simple $A$-module $S$  is $\gamma$-black if and only if $\psi\gamma(S) = S$.
\end{lem}

\begin{proof}
Suppose that $S$ is a $\gamma$-black simple $A$-module.
If $S$ is projective, then by definition $\psi\gamma(S)=S$.
If $\pd  S \geq 2$,
then  by Lemma ~\ref{lem:rin} we have $\gamma(S)=\soc\Omega^{2}(S)$
It follows from Lemma ~\ref{lem:mad}(1) that $\psi\gamma(S)=S$.

Suppose  $\psi\gamma(S)=S$.
If $S$ is projective, then $S$ is  $\gamma$-black.
If $S$ is not projective,
then it follows from Lemma ~\ref{lem:mad}(1) that
$\gamma(S)$ is a composition factor in $\Omega^{2}(S)$.
In particular,  $\Omega^{2}(S)$ is nonzero and thus $\pd S\geq 2$.
\end{proof}

Recall that a cycle in the resolution quiver of $A$ is called black  provided that
each vertex on this cycle is $\gamma$-black.

We have the following observation.

\begin{prop} \label{prop:bla-cyc}
Let $C$ be a cycle in the resolution quiver of $A$.
Then the following statements are equivalent.
\begin{enumerate}
  \item The vertices of $C$ form a $\psi$-cycle.
  \item Each vertex on $C$ is $\psi$-cyclic.
  \item $C$ is a black cycle.
\end{enumerate}
\end{prop}

\begin{proof}
(1) $\implies$ (2) This is obvious.

(2) $\implies$ (3) By Lemma ~\ref{lem:mad}(3)
each $\psi$-cyclic simple $A$-module is $\gamma$-black.

(3) $\implies$ (1)
Assume that the vertices of $C$
are $S_1,S_2,\cdots,S_m$ with $S_{i+1}=\gamma(S_i)$ for  $i\geq 1$.
Here, we identify $S_{m+i}$ with $S_i$.

Since $C$ is a black cycle,
each $S_i$ is $\gamma$-black and
each $\id S_i$ is not odd for $i \geq 1$.
It follows from Lemma ~\ref{lem:bla-mod} that
$\psi\gamma(S_i)=\psi(S_{i+1})=S_i$.

We claim that each $\pd  S_i$ is not odd and thus
each $S_i$ is  $\psi$-cyclic by Lemma ~\ref{lem:mad}(3).
Since $\psi(S_{i+1})=S_i$ for $i\geq 1$,
it follows that $S_m,S_{m-1},\cdots,S_1$ form a $\psi$-cycle.

For the claim,
suppose to the contrary that  the projective dimension of some $S_i$ is odd.
Then  by Lemma ~\ref{lem:mad}(2) $S_{i+1}$ is odd
and $\pd S_{i+1} \leq \pd S_i-2$.
By induction, we infer that $\pd S_{i+k} \leq \pd  S_i-2k$ for $k \geq 1$.
This is a contradiction.
\end{proof}

Suppose that global dimension of $A$ is infinite.
In particular, $A$ has no simple projective modules.
The following lemma describes  projective $A$-modules
of finite injective dimension and
injective $A$-modules of finite projective dimension.

\begin{lem} \label{lem:proj-inj}
Let $A$ be a connected Nakayama algebra with infinite global dimension, and let $P$ be an indecomposable projective $A$-module and $I$ be  an indecomposable injective $A$-module.
\begin{enumerate}
  \item The injective dimension of $P$ is infinite if and only if
  $P$ is a nontrivial submodule of $P(S)$ with $S$ a $\gamma$-cyclic simple $A$-module.
  \item The projective dimension of $I$ is infinite if and only if
  $I$ is a nontrivial factor module of $I(S)$ with $S$ a $\psi$-cyclic simple $A$-module.
\end{enumerate}
\end{lem}

\begin{proof}
(1) ``$\impliedby$"
If the injective dimension of $P$ is infinite,
then the injective dimension of its cosyzygy $\Omega^{-1}(P)$ is also finite.
It follows that  $\Omega^{-1}(P)$ contains at least one composition factor
with infinite injective dimension.

We claim that $\Omega^{-1}(P)$ contains at most one composition factor which is $\gamma$-cyclic.
It follows from Lemma ~\ref{lem:mad}(4) that
$\Omega^{-1}(P)$ contains precisely one composition factor which is $\gamma$-cyclic.
Since $P$ is a nontrivial submodule of $P(T)$ for each composition factor $T$ in $\Omega^{-1}(P)$,
we infer that $P$ is a nontrivial submodule of $P(S)$ with $S$ a $\gamma$-cyclic simple $A$-module.

For the claim, observe that $\soc P(T) = \soc P$ and $\gamma (T) = \gamma (\top P)$
for each composition factor $T$ in $\Omega^{-1}(P)$.
Since the composition factors in $\Omega^{-1}(P)$ have the same image under the map $\gamma$,
at most one of them is $\gamma$-cyclic.

``$\implies$"
Suppose that $P$ is a nontrivial submodule of $P(S)$ with $S$ a $\gamma$-cyclic simple $A$-module.
By the previous claim,
one can show that there exists only one composition factor in $\Omega^{-1}(P)$
which  is $\gamma$-cyclic, namely $S$.
Since the injective dimension of $S$ is infinite by Lemma ~\ref{lem:mad}(4),
the injective dimension of $\Omega^{-1}(P)$ is infinite.
Then the injective dimension of $P$ is infinite.

(2) This is dual to (1).
\end{proof}

Recall that an Artin algebra $A$ is called a {\em Gorenstein algebra} if  both
$\id A$ and $\pd DA$ are finite.
Here, $D$ is the usual dual for finitely generated $A$-modules.

We are now ready to prove Proposition ~\ref{prop:1.2}.
Indeed, using the maps $\gamma$ and $\psi$ as above,
several characterizations are given to decide whether
a connected Nakayama algebra with infinite global dimension is Gorenstein or not.
It strengthens \cite[Proposition  5(a)]{Rin2013}.

\begin{prop} \label{prop:gor}
Let $A$ be a connected Nakayama algebra  with infinite global dimension.
Then the following statements are equivalent.
\begin{enumerate}
  \item $A$ is a Gorenstein algebra.
  \item Each $\gamma$-cyclic simple $A$-module is $\gamma$-black.
  \item Each $\psi$-cyclic simple $A$-module is $\psi$-black.
  \item The set of $\gamma$-cyclic simple $A$-modules is exactly the set of $\psi$-cyclic simple $A$-modules.
\end{enumerate}
\end{prop}

\begin{proof}
(1) $\implies$ (2).
Let $S$ be a $\gamma$-cyclic simple $A$-module.
By Lemma ~\ref{lem:proj-inj}(1) the projective cover of $S$ has no nontrivial projective submodules.
Then the projective dimension of $S$ greater than $1$ and thus $S$ is $\gamma$-black.

Similarly, we have (1) $\implies$ (3).

(2) $\implies$ (4).
Following Proposition ~\ref{prop:bla-cyc} the set of $\gamma$-cyclic simple $A$-modules
is contained in the set of $\psi$-cyclic simple $A$-modules.
By \cite[Corollary 2.3]{She2014} the two sets have the same finite number of modules.
Then they must coincide.

(4) $\implies$ (2).
Since each $\psi$-cyclic simple $A$-module is $\gamma$-black,
it follows that each $\gamma$-cyclic simple $A$-module is $\gamma$-black.

Similarly, one can prove that (3) and (4) are equivalent.

(2) + (3) $\implies$ (1).
By Lemma ~\ref{lem:proj-inj} all projective $A$-modules have finite injective dimension and
all injective $A$-modules have finite projective dimension.
Then $A$ is a Gorenstein algebra.
\end{proof}

We mention that the global dimension condition in Proposition ~\ref{prop:gor} cannot be omitted;
see \cite[Example 2]{Rin2013}.

\section{Cartan matrices and resolution quivers} \label{sec:4}
In this section, we study the connection between the Cartan matrix and the resolution quiver
for a fixed connected  Nakayama algebra $A$.

Denote by $n = n(A)$ the number of non-isomorphic simple $A$-modules.
Take a complete set $\{S_1,S_2,\cdots,S_n\}$ of pairwise non-isomorphic simple $A$-modules.
The Cartan matrix $\bfC_A=(c_{ij})$ of $A$ is an $n \times n$ matrix,
where $c_{ij}$ is the number of copies of $S_i$
appearing in a composition series for the projective cover of $S_j$.

Recall that two $n\times n$ integer matrix $\bfX$ and $\bfY$ are {\em $\bbZ$-equivalent}
provided that there exist invertible integer matrices $\bfP$ and $\bfQ$ such that $\bfP\bfX\bfQ = \bfY$.
For an $n\times n$ integer matrix, its {\em Smith normal form} is the $n\times n$ diagonal integer matrix
\[\diag(d_1,d_2,\cdots,d_r,0,\cdots,0)\]
where $d_1,d_2,\cdots,d_r\in \bbN^+$ and $d_i$ divides $d_{i+1}$ for $1\leq i \leq r-1$.
The Smith normal form always exists and is unique;
it is $\bbZ$-equivalent to the original matrix.

Denote by $c(A)$ the number of cycles and
by $w(A)$ the weight of a cycle in the resolution quiver of $A$.

The following result provides a connection between the Cartan matrix
and the resolution quiver for $A$.

\begin{prop} \label{prop:smith}
Let $A$ be a connected Nakayama algebra.
Then the Smith normal form of its Cartan matrix $\bfC_A$ is the diagonal matrix
\[\diag(1,\cdots,1,w(A),0,\cdots,0)\] with $c(A)-1$ zeros on the diagonal.
In particular, the rank of $\bfC_A$ is $n(A)+1-c(A)$.
\end{prop}

\begin{proof}
Recall \eqref{eq:2.2} in Section ~\ref{sec:2}.
For $1\leq i\leq r-1$,
it is easy to show that the Cartan matrix
$\bfC_{A_i}$ of $A_i$ and the block diagonal matrix $\diag(1,\bfC_{A_{i+1}})$
are $\bbZ$-equivalent.

By Lemma ~\ref{lem:bij} we have $c(A_i)=c(A_{i+1})$ and $w(A_i)=w(A_{i+1})$.
Then it suffices to prove the assertion for the selfinjective algebra $A_r$.

Assume that $A$ is a connected selfinjective Nakayama algebra.
Denote by $m$ its radical length.
Then the admissible sequence of $A$ is $(m,m,\cdots,m)$.
It is routine to show that
the resolution quiver of $A$ consists of $\gcd(m,n)$ cycles of the same weight
$w=\frac{m}{\gcd(m,n)}$.
Here, `$\gcd$' is the greatest common divisor.

Let $m=kn+r$ with $k\in \bbN$ and $1\leq r\leq n$.
After possible permutations on simple $A$-modules,
the Cartan matrix $\bfC_A$ is a circulant matrix given by
\[c_{ij} =
\begin{cases}
k+1, & \mbox{ if $0 \leq i - j< r$ or $j - i > n - r$} \\
k,  & \mbox{ otherwise}.
\end{cases}\]
It follows from \cite{Wil2014} that
the Smith normal form of the Cartan matrix $\bfC_A$ is
the diagonal matrix $\diag(1,\cdots,1,w,0,\cdots,0)$
with $\gcd(m,n)-1$ zeros on the diagonal.
Then the rank of $\bfC_A$ is $n+1-\gcd(m,n)$.
This finishes our proof.
\end{proof}

\begin{rmk} \label{rmk:det}
Use the notation in the  Proposition ~\ref{prop:smith}.
\begin{enumerate}
  \item Following \cite[Theorem 6]{BFVZ1985},
  the  global dimension of $A$ is finite if and only if
  the determinant of the Cartan  matrix  $\bfC_A$ is $1$.
  In fact, by \cite[Proposition 2.2(5)]{CY2014} left retractions also preserve
  the determinants of Cartan matrices.
  Then  the determinant of the Cartan  matrix $\bfC_A$ is $w(A)$
  if the resolution quiver of $A$ is connected;
  see also \cite[Lemma 2]{BFVZ1985}.
  This provides another proof of Proposition ~\ref{prop:1.1}.
  \item  Denote by $A^\op$ the opposite algebra of $A$.
  Then the Cartan matrix of $A^\op$  is the transpose  of $\bfC_A$.
  Since $\bfC_A$ and its transpose $\bfC_A^T$ have the same Smith normal form,
  it follows that $c(A)=c(A^\op)$ and $w(A)=w(A^\op)$; see also \cite[Proposition 5.2]{She2015}.
\end{enumerate}
\end{rmk}

Let  $X$  be a subset of  $\{S_1,S_2,\cdots, S_n\}$.
There exists a $n\times 1$ vector $\xi_X$ associated with $X$,
where the $i$-th entry of $\xi_X$ is $1$ if $S_i$ is in $X$
and the $i$-th entry of $\xi_X$ is $0$ if $S_i$ is not in $X$ for $1\leq i\leq n$.
Denote by ${\bf 1}$ the $n\times 1$ vector $(1,\cdots, 1)^T$.

We have the following observation.

\begin{prop} \label{prop:linear}
Let $A$ be a connected Nakayama algebra.
Denote by $\Gamma$ the set of cycles and by $B\Gamma$
the set of black cycles in the resolution quiver of $A$.
\begin{enumerate}
  \item The vectors $\{\xi_C\}_{C\in \Gamma}$ are maximal linearly independent solutions
  to the linear system $\bfC_A \xi = w(A) {\bf 1}$.
  \item The vectors $\{\xi_C\}_{C\in \Gamma}$ are the entire nonnegative integer solutions
  to the linear system  $\bfC_A \xi = w(A) {\bf 1}$.
  \item The vectors $\{\xi_E\}_{E\in B\Gamma}$ are  maximal linearly independent solutions
  to the linear system $\bfC_A^T \xi = \bfC_A \xi = w(A) {\bf 1}$.
\end{enumerate}
\end{prop}

\begin{proof}
(1)  Since  $C$ is a cycle in the resolution quiver,
for $1\leq i\leq n$ the simple $A$-module $S_i$
appears exactly $w(A)$ times in the direct sum $\oplus_{S\in C}P(S)$.
It follows that $\bfC_A \xi_C = w(A) {\bf 1}$.
Since the vectors $\{\xi_C\}_{C\in \Gamma}$ have disjoint support,
they are linearly independent.
Then we obtain $c(A)$ linearly independent solutions to the linear system
$\bfC_A \xi = w(A) {\bf 1}$.
By Proposition ~\ref{prop:smith} the number of these solutions is $n + 1 - \rank \bfC_A$.
Therefore, the solutions  $\{\xi_C\}_{C\in \Gamma}$ are maximal.

(2) This follows from (1).

(3) Let $E$ be a cycle in the resolution quiver of $A$.
By Proposition ~\ref{prop:bla-cyc} the cycle   $E$ is
black if and only if the vertices of $E$ form a $\psi$-cycle.
By (2) the vertices of $E$ form a $\psi$-cycle if and only if
$\bfC_A^T \xi_E = w(A) {\bf 1}$.
Then the vectors $\{\xi_E\}_{E\in B\Gamma}$ are linearly independent solutions
to the linear system  $\bfC_A^T \xi = \bfC_A \xi = w(A) {\bf 1}$.

Denote by  $\Psi$ is the set of $\psi$-cycles for $A$.
We claim that the vectors $\{\xi_C\}_{C\in \Gamma\cup\Psi}$ are linearly independent.
For a solution $\xi$ to the desired linear system,
by (1) we have $\xi = \sum_{C\in \Gamma} a_{C}\xi_C=\sum_{D\in \Psi} b_D\xi_D$.
Observe that the intersection $\Gamma\cap \Psi$ is the set  of black cycles.
Since  the vectors $\{\xi_C\}_{C\in \Gamma\cup\Psi}$ are linearly  independent,
we have $\xi = \sum_{C\in B\Gamma} a_C\xi_C$.
This proves (3).

For the claim, let $\sum_{C\in \Gamma\cup\Psi} a_C \xi_C=0$.
For a non-black cycle $C \in \Gamma$,
it follows from Proposition ~\ref{prop:bla-cyc} that there exists
some $S_i$ on $C$  which is not $\psi$-cyclic.
Observe that the $i$-th entry of $\sum_{C\in \Gamma\cup\Psi} a_C \xi_D$ is $a_C$.
It follows that $a_C=0$ for each non-black cycle $C\in \Gamma$.
Since the vectors $\{\xi_C\}_{C \in \Psi}$ have disjoint support,
$a_C=0$ for each $\psi$-cycle $C$.
Therefore, the vectors $\{\xi_C\}_{C\in \Gamma\cup\Psi}$ are linearly independent.
\end{proof}

Denote by $b(A)$ be the number of black cycles in the resolution quiver of $A$.

We have the following.

\begin{cor}
Let $A$ be a connected Nakayama algebra.
\begin{enumerate}
  \item $b(A)$ is nonzero if and only if $\rank  \begin{pmatrix} \bfC_A^T  \\   \bfC_A \end{pmatrix} = \rank \begin{pmatrix} \bfC_A^T &{\bf 1} \\   \bfC_A  & {\bf 1}\end{pmatrix}$.
  \item If $b(A)$ is nonzero, then $b(A) = n(A) + 1 - \rank (\bfC_A,\bfC_A^T)$.
\end{enumerate}
\end{cor}

{\bf\noindent Acknowledgements.}
The author is very thankful to Professor Xiao-Wu Chen and Professor Guodong Zhou
for numerous  inspiring discussions.
This work is supported by China Postdoctoral Science Foundation (No. 2015M581563)
and Science and Technology Commission of Shanghai Municipality (No. 13DZ2260400).

\end{document}